\documentclass[11pt, reqno]{jmcs-t}
\usepackage{amsmath, amsthm, amscd, amsfonts, amssymb, graphicx, color}
\usepackage[backref,colorlinks=true]{hyperref}

\newtheorem{corollary}{Corollary}[section]

\newtheorem{lemma}{Lemma}[section]

\newtheorem{theorem}{Theorem}[section]

\theoremstyle{definition}
\theoremstyle{remark}

\usepackage{graphicx,dsfont,amssymb,booktabs}

\makeatletter
\let\@receivedat\relax
\makeatother

\begin{document}

\title{Faedo-Galerkin Approximations for the nonlinear Deterministic Constrained Modified Swift-Hohenberg Equation}

\author{Saeed Ahmed$^1$\corref{c1}}
\ead{saeed.msmaths21@iba-suk.edu.pk}
\author{Javed Hussain$^2$\corref{c2}}
\ead{javed.brohi@iba-suk.edu.pk}

\address{Department of Mathematics,
Sukkur IBA University, Pakistan$^{1,2}$\\}


\authors{S. Ahmed, J. Hussain}  

\date{\today}

\begin{abstract}

In this paper, we study the well-posedness of the nonlinear deterministic constrained modified Swift-Hohenberg equation; this equation belongs to class of amplitude equations which describe the appearance of pattern formation in nature. The proof for existence and uniqueness is based on the Faedo-Galerkin compactness method.

\end{abstract}

\maketitle

\baselineskip 12pt

\section{Introduction}

Doelman et al.\cite{Doelman} studied the following modified Swift-Hohenberg equation for the first time in 2003

\begin{align}{\label{dole}}
    u_{t}=-\alpha (1+ \Delta)^{2}u +\beta u-\gamma |\nabla u|^{2}-u^{3}
\end{align}
for a pattern formation with two unbounded spatial directions. Here, $\alpha >0, \beta, and ~\gamma $ are constants. When $\gamma=0$, equation (\ref{dole}) becomes the usual Swift-Hohenberg equation. The extra term $\gamma |\nabla u|^{2}$  reminiscent of the Kuramoto-Sivashinsky equation, which arises in
the study of various pattern formation phenomena involving some kind of phase
turbulence or phase transition \cite{Kuramoto, Siivashinsky} breaks the symmetry $u \rightarrow -u$.\\

From the formation of patterns to turbulence, the study of nonlinear differential equations has played a pivotal role in understanding physical phenomena. Of these equations, the Swift-Hohenberg equation has attracted scholars because it describes systems involving pattern-forming instabilities, such as Rayleigh–Benard convection and optical systems. Nonlinearity and constraints have been introduced to understand physical systems with more complex behaviors \cite{yochelis2006excitable, burke2006localized, kozyreff2006asymptotics}. More precisely, higher-order nonlinearities and domain constraints were introduced to comprehend localized patterns and complex spatiotemporal dynamics \cite{wetzel2010spatially, gomez2008swift, kirr2014pattern}. \\

Motivated by a broader range of physical systems, we are concerned with the following equation  obtained by modifying equation (\ref{dole}).
In the equation (\ref{SHEq}), $u(x,t)$ evolves under \textbf{negative bihormonic operator} $(-\Delta^{2}u)$, \textbf{ diffusion} $(2\Delta u)$, \textbf{a linear reaction term} $(au)$ and \textbf{a higher order non-linearity} $(u^{2n-1})$.

\begin{equation}{\label{SHEq}}
    u_{t}=-\Delta^{2}u+2\Delta u -au - u^{2n-1}
\end{equation}

In recent years, various techniques have been used to study the different types of modified Swift-Hohenberg equation, such as bifurcation analysis \cite{Xiao}, stability \cite{cross1993pattern, vSaarloos2003front, sakaguchi1996stable}, and proof of the existence of attractors. For instance, a global attractor was proven by \cite{Song,Xu}, the pull-back attractor was shown by \cite{Park}, the existence of a uniform attractor was presented by \cite{Polat}, and the numerical solution for this equation by the Fourier spectral method was given by \cite{Liu}. Especially, Reika FUKUIZUMI and others \cite{Reika} have constructed the solution of deterministic version of 2D Stochastic Anisotropic Swift–Hohenberg Equation using Faedo Galerkin approximation.  However, as far as we know, there are few works  concerning Faedo-Galerkin Approximations for Constrained Modified Swift-Hohenberg equation (\ref{SHEq}) directly.\\
It has been observed that the analytical solutions of various PDEs are difficult to establish. Therefore, in this scenario, many approximation techniques are used. Along with  Semi group theory and Sobolev spaces, the Faedo-Galerkin method is a potent tool for handling these types of higher-order PDEs \cite{temam2001navier, lions2012nonhomogeneous, robinson2001infinite}. The Faedo-Galerkin method is beneficial in PDE analysis for finding the approximate solution to the partial differential equation containing nonlinear terms and higher-order operators. In this method, the projection of an infinite-dimensional problem onto a finite-dimensional subspace enables the construction of an approximate solution that converges to the original solution as the dimension of the subspace increases. This method has been rigorously presented in \cite{Lions}.\\
J. Hussain \cite{Hussain} proved the Faedo-Galerkin approximation for the following nonlinear heat equation on a Hilbert Manifold.

\begin{equation}
    \begin{cases}
\frac{\partial u}{\partial t} = \pi_{u} \left(  \Delta u - u^{2n-1}\right) \\
u(0) = u_0.
\end{cases}
\end{equation}

In this study, we are motivated by the aforementioned observations, and we use the concept of J. Hussain \cite{Hussain} but in a more general sense. That is, by taking {bihormonic operator} $(\Delta^{2}u)$, {negative diffusion} $(-2\Delta u)$, {a linear reaction term} $(au)$ and {a higher order non-linearity} $(u^{2n-1})$.  Thus, we are concerned with the Faedo-Galerkin method for the following constrained modified Swift-Hohenberg equation to analyze the existence, uniqueness  and invariance of the solution on the Hilbert manifold.

\begin{align}{\label{main_Prb_1}}
        \frac{\partial u}{\partial t} &=\pi _{u}(-\Delta^{2}u
        +2\Delta u -au -u^{2n-1}) \\
u(0) &=u_{0}(x),~~~~~~~~~~\text{for}~~x \in  \mathcal{O}  \notag 
    \end{align}

 where $\mathcal{O} \subset \mathcal{R}^{d}$ denotes a bounded domain with a smooth boundary $\partial \mathcal{O}$,  $n \in \mathbb{N} $  (or, in a general sense, an actual number such that $n>\frac{1}{2}$), and $u_{0} \in H_{0}^{1}(\mathcal{O}) \cap H^{2}(\mathcal{O})\cap \mathcal{M}$. \\

Our paper is organized as follows. In Section 2, assumptions, definitions, functional settings, and estimates are introduced. In Section 3,  the compact estimates are presented. In Section 4, we present our main result of the problem and the well-posedness of the solution to the proposed problem.\\

We hope that this contribution will help explore a broader understanding of physical applications, especially those related to Physics, Engineering, and Applied Mathematics.

\section{Functional settings, Assumptions on Domain, Projection \& Manifold, and Some Important Estimates }

This section aims to include the functional settings, assumptions, key results, important definitions, and notations that help to prove our main problem.
\addtocontents{toc}{\protect\setcounter{tocdepth}{1}}

\subsection{{\textbf{Functional Settings}}}

We assume that $\mathcal{O} \subset R^{d}$ with a smooth boundary $\partial \mathcal{O}$.  $L^{p}(\mathcal{O})$, with $1\leq p \leq \infty$ is the Normed space with the norm

\begin{align*}
    \|u\|_{L^{p}(\mathcal{O})} = \left(\int_{\mathcal{O}}{(u(s))^{p}}~ds\right)^{\frac{1}{p}},~~~ where~~u \in \mathcal{O}
\end{align*}

Particularly, $L^{2}(\mathcal{O})$ is the Normed space with the norm\\

\begin{align*}
    \|u\|_{L^{2}(\mathcal{O})} = \left(\int_{\mathcal{O}}{(u(s))^{2}}~ds\right)^{\frac{1}{2}},~~~ where~~u \in \mathcal{O}
\end{align*}

This is a Hilbert space with the following inner product:

\begin{align*}
    \langle u, u' \rangle = \int_{\mathcal{O}}{u(s)~u'(s)}~ds ~~~~~~~~ u, ~u' \in \mathcal{O}
\end{align*}  
 The above definition follows from (\cite{zheng2004nonlinear}). Moreover, $H^{n}(\mathcal{O})$- {Sobolev space of order $n$} is defined as, for any $n\in N$ and $u \in L^{2}(\mathcal{O})$, the weak derivatives of $u$ up to order $n$ are also in $L^{2}(\mathcal{O})$. The Sobolev space defined above is a Banach space with the norm defined as
\begin{align}{\label{normsoblv}}
    \|u\|^{2}_{H^{n} (\mathcal{O})} &= \left( \sum_{|\alpha|\leq n} \|D^{\alpha}u\|^{2}_{L^{2}(\mathcal{O})} \right)\\
    \text{where}, \alpha&= (\alpha_{1},\alpha_{2},...\alpha_{n}), ~~|\alpha| = \alpha_{1}+\alpha_{2}+...+\alpha_{n}\\
    \text{and},~~~~ D^{\alpha}(u) &= \frac{\partial^{\alpha_{1}+\alpha_{2}+...+\alpha_{n}}u}{\partial x_{1}^{\alpha_{1}}\partial x_{2}^{\alpha_{2}}...\partial x_{n}^{\alpha_{n}}}
\end{align}
In particular, a Sobolev space of order one (\cite{zheng2004nonlinear}) is defined as
\begin{align*}
    H^{1}(\mathcal{O}) & = \left\{ u \in L^{2}(\mathcal{O})  : \text{f  has weak derivative}   \right\} \\
    \text{The norm on} ~~    H^{1}(\mathcal{O}) ~~\text{is given as: }\\
     \|u\|^{}_{H^{1} (\mathcal{O})} &= \sqrt{\|u\|^{2}_{L^{2} (\mathcal{O})}+ \|\nabla u\|^{2}_{L^{2} (\mathcal{O})}}
\end{align*}

The intersection of the two Banach spaces is, again, a Banach space. Following this, $ \mathcal{V} =H^{1}_{0} \cap H^{2}$ is a Banach space, with the norm defined as

\begin{align*}
    \|u\|^{}_{\mathcal{V}} &= \sqrt{\|u\|^{2}_{H^{1}_{0}} +  \|u\|^{2}_{H^{2}}} =\sqrt{\|u\|^{2}_{L^{2} (\mathcal{O})} +  2\|\nabla u\|^{2}_{L^{2} (\mathcal{O})} +   \|\Delta u\|^{2}_{L^{2} (\mathcal{O})} }
\end{align*} 

We define the self-adjoint operator (\ref{Self-Ad_op}) $A: D(A)  \longrightarrow L^{2}(\mathcal{O}) $ as:
\begin{align*}
    D(A)&= \left\{ u \in H^{1}_{0} \cap H^{2} \cap H^{4}; Au \in L^{2}(\mathcal{O})\right\}\\
Au &=  \Delta^{2}u - 2  \Delta u,~~~~~~~\forall u \in D(A)
    \end{align*} 

Assume that the spaces, $(\mathcal{E}, \|.\|) $, $(\mathcal{V}, \|.\|_{\mathcal{V}}) $  and $(\mathcal{H}, |.|) $ are Banach spaces and   

\begin{eqnarray} 
{\label{Ass_2.2.2}} 
\mathcal{H}&:=& \mathcal{L}^2(\mathcal{O}),  \quad \mathcal{V} := H_{0}^{1}(\mathcal{O}) \cap H^{2}(\mathcal{O}), \quad \mathcal{E} := D(A) = H_{0}^{1}(\mathcal{O}) \cap H^{2}(\mathcal{O}) \cap H^{4}(\mathcal{O}).
\end{eqnarray}  
Moreover,  the  inclusion $\mathcal{E} \subset \mathcal{V} \subset \mathcal{H}$ is both dense and continuous; that is,
\begin{align*}
 \mathcal{E} \hookrightarrow \mathcal{V} \hookrightarrow  \mathcal{H}
\end{align*}
In addition, the solution space is defined as:
\begin{equation*}
   X_{T}:=L^{2}\left( 0,T;\mathcal{E}\right) \cap C\left( \left[ 0,T\right] ;\mathcal{V}\right) ,
\end{equation*} 
Furthermore, $\left( X_{T},\left\vert \cdot\right\vert _{X_{T}}\right) $  is the Banach space with the following norm:
\begin{equation*}
\left\vert u\right\vert _{X_{T}}^{}=\sqrt{\underset{p\in \lbrack 0,T]}{\sup }\left\Vert u(p)\right\Vert ^{2}
+\int_{0}^{T}\left\vert u(p)\right\vert_{\mathcal{E}}^{2}dp}.
\end{equation*}

\subsection{\textbf{Assumptions on Domain}}

This subsection aims to derive the crucial relations that would lead to useful continuous embeddings. For this purpose, we begin with the Gagliardo-Nirenberg-Sobolev inequality. ~ (\cite{hussain2015analysis, zheng2004nonlinear}).

\begin{corollary}

If we choose  $  r = q = 2, j = 0, p = 2n , m = 2$  and  $d =2$  then the Gagliardo-Nirenberg-Sobolev inequality becomes:
\begin{align}{\label{U_p}}
    \|u\|_{L^{p}(R^{2})} \leq C\|u\|^{a}_{H^{2}_{0}} ~ \|u\|^{1-a}_{L^{2}(R^{d})}
\end{align}

 Using (\ref{normsoblv}) and (\ref{U_p}) with $p=2$ and in Chapter 10 of (\cite{haase2014functional}), it follows that

\begin{align}{\label{H^2_Embb}}
    H^{2}\hookrightarrow L^{2}, \quad  H^{2}_{0}\hookrightarrow L^{2}
\end{align}

Now, if we assume that $ \mathcal{V}= H^{2}\cap H_{0}^{1} \hookrightarrow H$, where $H$ is a Hilbert space, then by using (\ref{U_p}) it follows that

\begin{align}{\label{inq1}}
    \|u\|_{L^{p}(R^{2})} &\leq C^{a}\|u\|^{a}_{\mathcal{V}}\\
\end{align}
\text{By choosing}~~$p=2,4,8n-8$ \text{it follows}~~~ 
\begin{align}
    \mathcal{V} \hookrightarrow L^{4}, \quad \mathcal{V} \hookrightarrow L^{2}, \quad  \mathcal{V} \hookrightarrow L^{8n-8}
\end{align}

\end{corollary}

\subsection{ \textbf{Projection and Manifold}}
We will consider the following Hilbert Manifold throughout the paper by  assuming that $\mathcal{H}$ is a Hilbert space with inner product $\langle ~ \cdot, ~\cdot \rangle$
\begin{align*}
    \mathcal{M} =  \{~ h \in \mathcal{H}, \quad &|h|_{\mathcal{H}}^{2}=1~\}     
\end{align*}   
The above definition of a manifold follows from \cite{masiello1994variational}. For  $u~ \in ~\mathcal{H}$, tangent space $T_{u}\mathcal{M}$ is
\begin{align*}
    T_{u}\mathcal{M}= \{~ u': ~~\langle u', u \rangle = 0~ \},~~~~~~~u' \in \mathcal{H}
\end{align*}

\begin{lemma}
The orthogonal projection of $u'$ onto $u$, $ \pi_{u}~: ~ \mathcal{H} \longrightarrow T_{u}\mathcal{M}$ is defined as
\begin{equation}{\label{lemma_Tangent}}
    \pi_{u}(u')= u'-\langle u', u \rangle ~u, ~~~~~~u' \in \mathcal{H}
\end{equation}  

If $ u \in \mathcal{E} \cap \mathcal{M}$ then using (\ref{lemma_Tangent}),  the projection of $-\Delta^{2}u+2\Delta u -au - u^{2n-1}$ under the map $\pi_{u}$ can be found as:\\

$ \pi _{u}(-\Delta^{2}u+2\Delta u -au - u^{2n-1})$
 \begin{align}
     &=-\Delta^{2}u+2\Delta u -au - u^{2n-1}
+\langle \Delta^{2}u-2\Delta u +au + u^{2n-1}, u \rangle ~u \notag \\
&=-\Delta^{2}u+2 \Delta u -au - u^{2n-1} + \langle \Delta^{2}u, u \rangle ~u -2\langle \Delta u, u \rangle ~u \notag  ~+a\langle u, u  
\rangle ~u+\langle u^{2n-1}, u \rangle ~u \notag \\
&=-\Delta^{2}u+2 \Delta u -au - u^{2n-1} + \langle \Delta u,  \Delta u \rangle ~u -2\langle - \nabla u, \nabla u \rangle ~u \notag  ~+a\langle u, u  
\rangle ~u+\langle u^{2n-1}, u \rangle ~u \notag \\
&=-\Delta^{2}u+2 \Delta u -au - u^{2n-1} + \| \Delta u\|^{2}_{{L}^{2}(\mathcal{O})} ~u + 2\| \nabla u\|^{2}_{{L}^{2}(\mathcal{O})} ~u \notag  ~ +au+\| u\|^{2n}_{{L}^{2n}(\mathcal{O})} u \notag \\ 
&=-\Delta^{2}u+2 \Delta u  + \| u\|^{2}_{{H}^{2}_{0}} ~u + 2\|  u\|^{2}_{{H}^{1}_{0}} ~u \notag  ~ +\| u\|^{2n}_{{L}^{2n}} u- u^{2n-1}  
 \end{align}   
Therefore, the projection of $-\Delta^{2}u+2\Delta u +au + u^{2n-1}$ under map  $\pi_{u}$ is
\begin{eqnarray} {\label{Projection_U}}
      \pi _{u}(-\Delta^{2}u+2\Delta u-au - u^{2n-1}) \notag =  -  \Delta^{2}u+2 \Delta u  + \|  u\|^{2}_{{H}^{2}_{0}} ~u + 2\|    u\|^{2}_{{H}^{1}_{0}} ~u  +\| u\|^{2n}_{{L}^{2n}} u- u^{2n-1} 
\end{eqnarray} 
\end{lemma}

\addtocontents{toc}{\protect\setcounter{tocdepth}{1}}

The following theorem is the main result of this study:

\begin{theorem}{\label{main_prob}}

Let  $\mathcal{E} \subset\mathcal{V} \subset \mathcal{H}$ satisfy the assumption (\ref{Ass_2.2.2}) and  $u_{0} \in \mathcal{V} \cap \mathcal{M}$, there is a unique solution $ u : [0, \infty) \longrightarrow \mathcal{V}$, where $u \in X_{T}$, to the problem

\begin{align}{\label{PB}}
        \frac{\partial u}{\partial t} &=\pi _{u}(-\Delta^{2}u+2\Delta u -au - u^{2n-1}) = -Au + F(u(t)) \notag\\ 
u(0) &=u_{0}.  
    \end{align}
Where   ~~~~   $ F(u)=\|  u\|^{2}_{{H}^{2}_{0}} ~u + 2\|    u\|^{2}_{{H}^{1}_{0}} ~u  +\| u\|^{2n}_{{L}^{2n}} u- u^{2n-1}  $  ~~~    and  ~~~    $ A = \Delta^{2}-2\Delta$
\end{theorem}

\subsection{\textbf{Some Important Estimates}}

In this subsection, we aim to show that the map $ F:  \mathcal{V}\longrightarrow \mathcal{H}$  defined as $F(u)=\|  u\|^{2}_{{H}^{2}_{0}} ~u + 2\|    u\|^{2}_{{H}^{1}_{0}} ~u  +\| u\|^{2n}_{{L}^{2n}} u- u^{2n-1} $ is locally Lipschitz by using (\ref{Ass_2.2.2}) and some fundamental inequalities. The following inequalities follow directly from  \cite{hussain2015analysis}.

\begin{lemma}
    
 For any $\alpha_{1} , \alpha_{2} \in \mathbb{R}$ and $n \in \mathbb{N}$, we have that

\begin{align}
    |\alpha_{1}^{n}-\alpha_{2}^{n}| &\leq n \alpha_{1}^{n-1} |\alpha_{1}-\alpha_{2}|
\end{align} 
{and}
\begin{align}{\label{ine}}
    |\alpha_{1}^{n}-\alpha_{2}^{n}|&\leq \frac{n}{2}  |\alpha_{1}-\alpha_{2}| \left(|\alpha_{1}|^{n-1}+ |\alpha_{2}|^{n-1}\right)
\end{align}
\end{lemma}

\begin{lemma}{\label{lemma_abt_liptz}}
Let $\mathcal{E}$, $\mathcal{V}$ and $\mathcal{H}$ follows the assumption (\ref{Ass_2.2.2}), and $ F :  \mathcal{V}\longrightarrow \mathcal{H}$ be a map defined as 
\begin{eqnarray*}
    F(u)=\|  u\|^{2}_{{H}^{2}_{0}} ~u + 2\|    u\|^{2}_{{H}^{1}_{0}} ~u  +\| u\|^{2n}_{{L}^{2n}} u- u^{2n-1} 
\end{eqnarray*}
Then, $F$ is locally Lipschitz; that is, there exists a constant $K$ such that for all $u_{1},u_{2} \in \mathcal{V}$ it follows that
\begin{eqnarray}{\label{liptz}}
    |F(u_{1})-F(u_{2})|_{\mathcal{H}} \leq K
    \begin{bmatrix} \left(\|u_{1}\|_{\mathcal{V}}^{2}+\|u_{2}\|_{\mathcal{V}}^{2}+\|u_{1}\|_{\mathcal{V}}\|u_{2}\|_{\mathcal{V}}\right)  +  \left( \|u_{1}\|_{\mathcal{V}}^{2n-1}+\|u_{2}\|_{\mathcal{V}}^{2n-1}\right) \\ (\|u_{1}\|_{\mathcal{V}}+\|u_{2}\|_{\mathcal{V}})  +  \left( \|u_{1}\|_{\mathcal{V}}^{2n}+\|u_{2}\|_{\mathcal{V}}^{2n}\right)+ \left( 1+\|u_{1}\|_{\mathcal{V}}^{2}+\|u_{2}\|_{\mathcal{V}}^{2}\right)^{\frac{1}{3}} 
    \end{bmatrix}  \|u_{1}-u_{2}\|_\mathcal{V}
\end{eqnarray} 
\end{lemma} 
\begin{proof}
Suppose that $  F(u)=\|  u\|^{2}_{{H}^{2}_{0}} ~u + 2\|    u\|^{2}_{{H}^{1}_{0}} ~u  +\| u\|^{2n}_{{L}^{2n}} u- u^{2n-1} = \sum _{i=1}^{4} F_{i}(u)$.

By considering $u_{1}$ and $u_{2}$ as fixed elements in $\mathcal{V}$  we evaluate each $F_{i}'s$ for estimation. 
 Let us take the first term $F_{1}(u)=\|  u\|^{2}_{{H}^{2}_{0}} ~u $

 \begin{eqnarray*}{}
     \left|F_{1}(u_{1})-F_{1}(u_{2})\right|_{\mathcal{H}} 
 &\leq& \|  u_{1}\|^{2}_{{H}^{2}_{0}} \left|u_{1}-u_{2}\right|_{\mathcal{H}} + \left( \|  u_{1}\|^{2}_{{H}^{2}_{0}} -\|  u_{2}\|^{2}_{{H}^{2}_{0}} \right)\left|u_{2}\right|_{\mathcal{H}}  \notag \\
     &\leq&  \left(\|  u_{1}\|^{2}_{{H}^{2}_{0}} + \|  u_{2}\|^{2}_{{H}^{2}_{0}}\right) \left|u_{1}-u_{2}\right|_{\mathcal{H}} + \left( \|  u_{1}\|_{{H}^{2}_{0}} +\|  u_{2}\|_{{H}^{2}_{0}} \right) \|u_{1}-u_{2}\|_{{H}^{2}_{0}} \left(\left|u_{1}\right|_{\mathcal{H}}+\left|u_{2}\right|_{\mathcal{H}}\right)  \notag \\
     \end{eqnarray*}

It is known that $ {H}^{2}_{0} \hookrightarrow \mathcal{H}$ and $ \mathcal{V} \hookrightarrow {H}^{2}_{0} $ are continuous, and  there are constants  $C_{1}>0$ and $C_{2}>0$ such that  $ |u|_{\mathcal{H}} \leq C _{1}\|u\|_{{H}^{2}_{0}}$ and $ |u|_{{H}^{2}_{0} } \leq C_{2} \|u\|_{\mathcal{V}}$. In addition, by taking $ k_{1}=2C_{1} C_{2}^{3}$, we can deduce that

\begin{eqnarray}{\label{F_{1}}}
     \left|F_{1}(u_{1})-F_{1}(u_{2})\right|_{\mathcal{H}} &\leq& C_{1} C_{2}^{3}\left( \left(\|  u_{1}\|^{2}_{\mathcal{V}} + \|  u_{2}\|^{2}_{\mathcal{V}}\right)   + \left( \|  u_{1}\|_{\mathcal{V}} +\|  u_{2}\|_{\mathcal{V}} \right)^{2}\right) \|u_{1}-u_{2}\|_{\mathcal{V}} \notag \\ 
     &=& 2C_{1} C_{2}^{3}\left( \|  u_{1}\|^{2}_{\mathcal{V}} + \|  u_{2}\|^{2}_{\mathcal{V}}   +  \|  u_{1}\|_{\mathcal{V}} \|  u_{2}\|_{\mathcal{V}} \right) \|u_{1}-u_{2}\|_{\mathcal{V}} \notag \\
 &\leq& k_{1}\left( \|  u_{1}\|^{2}_{\mathcal{V}} + \|  u_{2}\|^{2}_{\mathcal{V}}   +  \|  u_{1}\|_{\mathcal{V}} \|  u_{2}\|_{\mathcal{V}} \right) \|u_{1}-u_{2}\|_{\mathcal{V}}
     \end{eqnarray}

Next, we consider equation $F_{2}=2\|    u\|^{2}_{{H}^{1}_{0}} ~u $.\\

\begin{eqnarray*}
       \left|F_{2}(u_{1})-F_{2}(u_{2})\right|_{\mathcal{H}}
 &\leq& 2\|  u_{1}\|^{2}_{{H}^{1}_{0}} \left|u_{1}-u_{2}\right|_{\mathcal{H}} + \left( \|  u_{1}\|^{2}_{{H}^{1}_{0}} -\|  u_{2}\|^{2}_{{H}^{1}_{0}} \right)\left|u_{2}\right|_{\mathcal{H}}  \notag \\
     &\leq& 2 \left(\|  u_{1}\|^{2}_{{H}^{2}_{0}} + \|  u_{2}\|^{2}_{{H}^{1}_{0}}\right) \left|u_{1}-u_{2}\right|_{\mathcal{H}}  + \left( \|  u_{1}\|_{{H}^{1}_{0}} +\|  u_{2}\|_{{H}^{1}_{0}} \right) \|u_{1}-u_{2}\|_{{H}^{1}_{0}} \left(\left|u_{1}\right|_{\mathcal{H}}+\left|u_{2}\right|_{\mathcal{H}}\right).  \notag \\
     \end{eqnarray*}

 As $ {H}^{1}_{0} \hookrightarrow \mathcal{H}$ and $ \mathcal{V} \hookrightarrow {H}^{1}_{0} $ are continuous so there exit constants $C_{3}>0$ and  $C_{4}>0$ such that  $ |u|_{\mathcal{H}} \leq C _{3}\|u\|_{{H}^{1}_{0}}$ and  $ |u|_{{H}^{1}_{0} } \leq C_{4} \|u\|_{\mathcal{V}}$. In addition, by taking $k_{2}=4C_{3} C_{4}^{3}$ we can infer that:

\begin{eqnarray}{\label{F_2}}
     \left|F_{2}(u_{1})-F_{2}(u_{2})\right|_{\mathcal{H}} &\leq& 2C_{3} \left( \left(\|  u_{1}\|^{2}_{{H}^{1}_{0}} + \|  u_{2}\|^{2}_{{H}^{1}_{0}}\right)   + \left( \|  u_{1}\|_{{H}^{1}_{0}} +\|  u_{2}\|_{{H}^{1}_{0}} \right)^{2}\right) \|u_{1}-u_{2}\|_{{H}^{1}_{0}} \notag \\
     &\leq& 2C_{3} C_{4}^{3}\left( \left(\|  u_{1}\|^{2}_{\mathcal{V}} + \|  u_{2}\|^{2}_{\mathcal{V}}\right)   + \left( \|  u_{1}\|_{\mathcal{V}} +\|  u_{2}\|_{\mathcal{V}} \right)^{2}\right) \|u_{1}-u_{2}\|_{\mathcal{V}} \notag \\ 
     &=& k_{2}\left( \|  u_{1}\|^{2}_{\mathcal{V}} + \|  u_{2}\|^{2}_{\mathcal{V}}   +  \|  u_{1}\|_{\mathcal{V}} \|  u_{2}\|_{\mathcal{V}} \right) \|u_{1}-u_{2}\|_{\mathcal{V}}
\end{eqnarray}

Again, we use the function  $F_{3}(u) = \| u\|^{2n}_{{L}^{2n}} u$ for the third estimation.

  \begin{eqnarray*}
     \left|F_{3}(u_{1})-F_{3}(u_{2})\right|_{\mathcal{H}}       
 &\leq&  \left( \|  u_{1}\|^{2n}_{{L}^{2n}} -\|  u_{2}\|^{2n}_{{L}^{2n}} \right)\left|u_{1}\right|_{\mathcal{H}} + \|  u_{2}\|^{2n}_{{L}^{2n}} \left|u_{1}-u_{2}\right|_{\mathcal{H}}  \notag \\
\end{eqnarray*}

Using inequality (\ref{ine}) and the fact that $ \mathcal{V} \hookrightarrow {L}^{2n-1} $, $ \mathcal{V} \hookrightarrow \mathcal{H} $ and $ \mathcal{V} \hookrightarrow {L}^{2n} $ are continuous. By taking $k_{n}= max \{C^{n+1}, C^{n+1}\left( \frac{2n-1}{2}\right) \}$ we can deduce that

\begin{align} {\label{F_{3}}}
     \left|F_{3}(u_{1})-F_{3}(u_{2})\right|_{\mathcal{H}} &\leq  \left( \frac{2n-1}{2}\right) \left( \|  u_{1}\|^{2n-1}_{{L}^{2n}} -\|  u_{2}\|^{2n-1}_{{L}^{2n}} \right)\left|u_{1}\right|_{\mathcal{H}} + \|  u_{2}\|^{2n}_{{L}^{2n}} \left|u_{1}-u_{2}\right|_{\mathcal{H}}  \notag \\ 
      &\leq  C^{n+1} \left( \frac{2n-1}{2}\right) \left( \|  u_{1}\|^{2n-1}_{\mathcal{V}} -\|  u_{2}\|^{2n-1}_{\mathcal{V}} \right)\|u_{1}\|_{\mathcal{V}} + C^{n+1}\|  u_{2}\|^{2n}_{\mathcal{V}} \|u_{1}-u_{2}\|_{\mathcal{V}}    \notag \\ 
     &\leq k_{n}\left(\left( \|  u_{1}\|^{2n-1}_{\mathcal{V}} +\|  u_{2}\|^{2n-1}_{\mathcal{V}} \right)\left(\|u_{1}\|_{\mathcal{V}} 
 +  \|u_{2}\|_{\mathcal{V}}\right)+ \left(\|  u_{1}\|^{2n}_{\mathcal{V}} + \|  u_{2}\|^{2n}_{\mathcal{V}}  \right) \right) \|u_{1}-u_{2}\|_{\mathcal{V}}  
\end{align}

Finally, to prove the estimation for $F_{4}(u)=-u^{2n-1}$, we require the following inequality:

\begin{align}{\label{inq2}}
\left| |u_{1}|_{\mathcal{H}}^{2n-2} u_{1} - |u_{2}|_{\mathcal{H}}^{2n-2} u_{2} \right|_{\mathcal{H}} \leq c \left( |u_{1}|_{\mathcal{H}}^{2n-2} + |u_{2}|_{\mathcal{H}}^{2n-2} \right) |u_{1} - u_{2}|_{\mathcal{H}}
\end{align}

We take steps to prove inequality (\ref{inq2}). First, we suppose that $|u_{1}|_{\mathcal{H}}, ~|u_{2}|_{\mathcal{H}} \leq 1$ then, differentiation at zero of the function $u \rightarrow |u|^{p} $, where $p>1$, or one-sided differentiability at zero for $ p = 1$ gives the following equation:

\begin{align}{\label{sup}}
    \sup_{\substack{u_{1} \neq u_{2}, \\ |u_{1}|_{\mathcal{H}}, |u_{2}|_{\mathcal{H}} \leq 1}} \frac{\left| |u_{1}|_{\mathcal{H}}^{2n-2} u_{1} - |u_{2}|_{\mathcal{H}}^{2n-2} u_{2}\right|_{\mathcal{H}}}{|u_{1} - u_{2}|_{\mathcal{H}}} =: C_0 < \infty
\end{align}

For the case  $|u_{1}|_{\mathcal{H}} \text{ or } ~|u_{2}|_{\mathcal{H}} > 1$, we can continue as follows:\\

$  \frac{\left| |u_{1}|_{\mathcal{H}}^{2n-2} u_{1} - |u_{2}|_{\mathcal{H}}^{2n-2} u_{2}\right|_{\mathcal{H}}}{|u_{1} - u_{2}|_{\mathcal{H}}} = 
\left(|u_{1}|_{\mathcal{H}}+|u_{2}|_{\mathcal{H}}\right)^{2n-2} \left(\frac{\left| \left|\frac{u_{1}}{|u_{1}|_{\mathcal{H}} +|u_{2}|_{\mathcal{H}}}\right|_{\mathcal{H}}^{2n-2} \left(\frac{u_{1}}{|u_{1}|_{\mathcal{H}} +|u_{2}|_{\mathcal{H}}}\right) - \left|\frac{u_{2}}{|u_{1}|_{\mathcal{H}} +|u_{2}|_{\mathcal{H}}} \right|_{\mathcal{H}}^{2n-2} \left(\frac{u_{2}}{|u_{1}|_{\mathcal{H}} +|u_{2}|_{\mathcal{H}}}\right)\right|_{\mathcal{H}}}{ \left|\frac{u_{1}}{|u_{1}|_{\mathcal{H}} +|u_{2}|_{\mathcal{H}}} - \frac{u_{2}}{|u_{1}|_{\mathcal{H}} +|u_{2}|_{\mathcal{H}}}\right|_{\mathcal{H}}}\right) $

Taking $L_{1}=|u_{1}|_{\mathcal{H}}$, $L_{2}=|u_{2}|_{\mathcal{H}}$, $u'_{1}= \frac{u_{1}}{L_{1}+L_{2}}$,  and $u'_{2}= \frac{u_{2}}{L_{1}+L_{2}}$, the above equation reduces to:

\begin{align*}
     \frac{\left| |u_{1}|_{\mathcal{H}}^{2n-2} u_{1} - |u_{2}|_{\mathcal{H}}^{2n-2} u_{2}\right|_{\mathcal{H}}}{|u_{1} - u_{2}|_{\mathcal{H}}} = 
\left(L_{1}+L_{2}\right)^{2n-2} \frac{\left| \left|u'_{1}\right|_{\mathcal{H}}^{2n-2} u'_{1} - \left|u'_{2} \right|_{\mathcal{H}}^{2n-2} u'_{2}\right|_{\mathcal{H}}}{ \left|u'_{1} - u'_{2}\right|_{\mathcal{H}}}
\end{align*}

Using (\ref{sup}), it follows that

\begin{align*}
     \frac{\left| |u_{1}|_{\mathcal{H}}^{2n-2} u_{1} - |u_{2}|_{\mathcal{H}}^{2n-2} u_{2}\right|_{\mathcal{H}}}{|u_{1} - u_{2}|_{\mathcal{H}}} &\leq C_{0}
\left(L_{1}+L_{2}\right)^{2n-2} = C_{0}
\left(|u_{1}|_{\mathcal{H}}+|u_{2}|_{\mathcal{H}}\right)^{2n-2}\\
&\leq C_{n}
\left(|u_{1}|_{\mathcal{H}}^{2n-2}+|u_{2}|_{\mathcal{H}}^{2n-2}\right)
\end{align*}

Therefore, we used the claim in (\ref{inq2}). Our final estimation can be deduced using (\ref{ine}) and the Hölder inequality as follows:

\begin{eqnarray*}
     \left|F_{4}(u_{1})-F_{4}(u_{2})\right|_{\mathcal{H}}^{2} &=& \left|-u_{1}^{2n-1}+u_{2}^{2n-1}\right|_{\mathcal{H}}^{2} \leq  \left|u_{1}^{2n-2}-u_{2}^{2n-2}\right|_{\mathcal{H}}^{2} \left| u_{1}-u_{2}\right|_{\mathcal{H}}^{2}\notag \\
 &\leq& C^{2}_{n} \int_{\mathcal{O}} \left| u_{1}(s) - u_{2}(s) \right|_{\mathcal{H}}^{2} \left( \left| u_{1}(s) \right|_{\mathcal{H}}^{2n-2} + \left| u_{2}(s) \right|_{\mathcal{H}}^{2n-2} \right)^{2} ~ds \\
&\leq& C^{2}_{n} \left( \int_{\mathcal{O}} \left| u_{1}(s) - u_{2}(s) \right|_{\mathcal{H}}^{6} ~ds \right)^{\frac{1}{3}} 
\left( \int_{\mathcal{O}} \left( \left| u_{1}(s) \right|_{\mathcal{H}}^{2n-2} + \left| u_{2}(s) \right|_{\mathcal{H}}^{2n-2} \right)^{3} ~ds \right)^{\frac{2}{3}} \notag \\
\end{eqnarray*}

The following inequality can be obtained using the Minkowski inequality:

\begin{align}
  \left|F_{4}(u_{1})-F_{4}(u_{2})\right|_{\mathcal{H}} &\leq C_{n}\| u_{1} - u_{2} \|_{L^{6}(\mathcal{O})}
     \left( \left( \int_{\mathcal{O}} \left|u_{1}(s)\right|_{\mathcal{H}}^{6n-6} ~ds \right)^{\frac{1}{3}} + \left( \int_{\mathcal{O}} \left|u_{2}(s)\right|_{\mathcal{H}}^{6n-6} ~ds \right)^{\frac{1}{3}} \right) 
\end{align}

We know that:

\begin{align*}
    \left|u_{}(s)\right|_{\mathcal{H}}^{6n-6} \leq \max \{1, |u(s)|_{\mathcal{H}}\}^{6n-6}
\end{align*}

Therefore,

\begin{align*}
    \left( \int_{\mathcal{O}} \left|u_{}(s)\right|_{\mathcal{H}}^{6n-6} ~ds \right)^{\frac{1}{3}} \leq \begin{cases}
\left( |\mathcal{O}| + |u(s)|_{\mathcal{L}^{6n}(\mathcal{O})}^2 \right)^{1/3}, & d = 2 \\
\left( |\mathcal{O}| + |u(s)|_{\mathcal{L}^{6}(\mathcal{O})}^2 \right)^{1/3}, & d = 3
\end{cases}
\end{align*}

Hence, for $d=2$, and $d=3$ using the continuity of embedding $\mathcal{V} \hookrightarrow \mathcal{L}^{6n}(\mathcal{O})$ and  $\mathcal{V} \hookrightarrow \mathcal{L}^{6}(\mathcal{O})$ respectively, we can deduce that there is a constant $k'_{n}$ such that:

\begin{align}{\label{f-4}}
    \left|F_{4}(u_{1})-F_{4}(u_{2})\right|_{\mathcal{H}}   &\leq   
    k'_{n} 
     \left( 1+\|u_{1}\|_{\mathcal{V}}^{2} + \|u_{2}\|_{\mathcal{V}}^{2} \right)^{\frac{1}{3}} \| u_{1} - u_{2} \|_{\mathcal{V}}
\end{align}

By choosing $K =  max \left\{ k_{1},k_{2}, k_{n}, k'_{n}\right \}$ and using  estimates (\ref{F_{1}}), (\ref{F_2}), (\ref{F_{3}}), and (\ref{f-4}), we obtain (\ref{liptz}).  Therefore, $F$ is a local Lipschitz function.

\end{proof}


\begin{lemma}{\label{Self-Ad_op}} 
    The operator $A: D(A) \longrightarrow \mathcal{V}$ defined by 
\begin{align*}
    A &= \Delta^{2}-2\Delta 
\end{align*}
 is the self-adjoint operator.\\
\begin{proof}
    Suppose $u_{1}$ and $u_{2}$ are the elements in $D(A)$ then 
\begin{align*}
\langle Au_{1}, u_{2}  \rangle 
        &= \langle \Delta^{2} u_{1}-2\Delta u_{1} , u_{2}  \rangle = \langle \Delta^{2}u_{1} , u_{2}  \rangle - 2 \langle \Delta u_{1} , u_{2}  \rangle =  \langle \Delta u_{1} , \Delta u_{2}  \rangle + 2 \langle \nabla u_{1} , \nabla u_{2}  \rangle\\
        &=\langle  u_{1} , \Delta ^{2} u_{2}  \rangle - 2 \langle  u_{1} , \Delta u_{2}  \rangle = \langle  u_{1} , \left(\Delta ^{2} -2 \Delta \right)u_{2} \rangle = \langle u_{1}, Au_{2}  \rangle
\end{align*}  

Thus, operator $A$ is the self-adjoint operator.
\end{proof}

\end{lemma}

\begin{lemma}{\label{Lemma_weak derivative_Abs}}
    \cite{temam2001navier}, \cite{hussain2015analysis} Suppose  $\mathcal{E}, ~~\mathcal{V} , ~~ \mathcal{H}$ satisfy the assumption (\ref{Ass_2.2.2}) then the dual  $\mathcal{H'}$ of $\mathcal{H}$ and dual $\mathcal{V'}$ of $\mathcal{V}$ satisfy the relation: 
    \begin{align*}
         \mathcal{H} \hookleftarrow \mathcal{V} \cong \mathcal{V'} \hookleftarrow \mathcal{H'}
    \end{align*}

Additionally, if the abstract function $u$ is in $L^{2}\left ( 0, T; \mathcal{V}\right)$, the weak derivative $\frac{\partial u}{ \partial t}$   is in $L^{2}\left ( 0, T; \mathcal{V'}\right)$. \\
Furthermore, 
\begin{align*}
    |u(t)|^{2} = |u_{0}|^{2}+ 2 \int_{0}^{t}{ \langle u'(s), u(s) \rangle}~ds,  ~~~~~~~t \in [0,T]~~~~~~~- a.e
\end{align*} 

\end{lemma}

\textbf{Remarks}
We can identify  $u$ as $\tilde{u}$ if, for $u$, the assumptions of Lemma (2.4) are fulfilled. Moreover,
a weak derivative of $\tilde{u}$ exists and is equal to $\frac{\partial u}{ \partial t}$.

\begin{section}{The Compactness Estimates}

Assume that $\mathcal{H}$ has an orthonormal basis $\{e_{j}\}$ that contains the eigenvectors of $\mathcal{A}$ and $\{\lambda_{j}\}$ is the set of eigenvalues of $\mathcal{A}$. Suppose $\mathcal{H}_{k}$ is the space generated by $\{e_{j}\}_{j=1}^{k}$. 
Clearly, for $l\geq k$ we have $\mathcal{H}_{k} \subset \mathcal{H}_{l} $. The linear operator can be expressed as follows:
\begin{eqnarray}
    \pi _{k}u := \sum_{j=1}^{k} \langle u , e_{j} \rangle e_{j}
 \end{eqnarray}

Here, we address the problems mentioned below.

\begin{eqnarray}{\label{approimate-problem}}
     \frac{d u_{k}}{d t} &= -A_{k}u_{k} + F_{k}(u_{k}(t)) \\ 
u_{k}(0) &=\frac{\pi_{k}u_{0}}{|\pi_{k}u_{0}|_{\mathcal{H}}} \notag  
\end{eqnarray}

where $A_{k}(\cdot)= \pi_{k}A(\cdot)$ and $F_{k}(\cdot)= \pi_{k}F(\cdot)$.

In the following steps, we intend to obtain some estimates for the approximated solution $u_{k} = \sum_{j=1}^{k} g_{jk} e_{k}$

As we know that:

\begin{eqnarray*}
    \frac{d}{dt}\left(\frac{|u_{k}(t)|_{\mathcal{H}}^{2}-1}{2}\right)= \left\langle u_{k}'(t), u_{k}(t) \right\rangle
\end{eqnarray*}

Now, from equation (\ref{approimate-problem}), we can deduce that:

\begin{align*}
    \frac{d}{dt}\left(\frac{|u_{k}(t)|_{\mathcal{H}}^{2}-1}{2}\right) &= \left\langle -A_{k}u_{k} + F_{k}(u_{k}), u_{k}\right\rangle \\
   &= {\langle -\Delta^{2} u_{k}+2\Delta u_{k} , u_{k}  \rangle }
    + { \langle \|  u_{k}\|^{2}_{{H}^{2}_{0}} ~u_{k} + 2\|    u_{k}\|^{2}_{{H}^{1}_{0}} ~u_{k}  +\| u_{k}\|^{2n}_{{L}^{2n}} u_{k}- u_{k}^{2n-1}, u_{k} \rangle}  \\
&=  { \langle -\Delta^{2} u_{k} ,u_{k} \rangle} + { \langle 2\Delta u_{k} ,u_{k} \rangle}  + { \langle\|  u_{k}\|^{2}_{{H}^{2}_{0}} ~u_{k} ,u_{k} \rangle} + { \langle 2\|    u_{k}\|^{2}_{{H}^{1}_{0}} ~u_{k},u_{k} \rangle} + { \langle \| u\|^{2n}_{{L}^{2n}} u_{k} ,u_{k} \rangle} \\
    &+ { \langle -u_{k}^{2n-1} ,u_{k} \rangle}
\end{align*}

As ~~~~$ \frac{\partial u_{k}}{\partial t} =\pi _{u_{k}}(-\Delta^{2}u_{k}+2\Delta u_{k} -au_{k} - u_{k}^{2n-1})  \in \mathcal{M}$ ~~~ so~~~ $ \langle \frac{\partial u_{k}}{\partial t}, u_{k} \rangle = 0$~~~ \\
In addition, for ~~~ $ \langle u_{k} , u_{k}\rangle = |u_{k}|^{2}_{\mathcal{H}}$, it follows that 

\begin{align*}
  \frac{d}{dt}\left(\frac{|u_{k}(t)|_{\mathcal{H}}^{2}-1}{2}\right) &=  {\left(\|u_{k}\|^{2}_{H^{2}_{0}} + 2 \|u_{k}\|^{2}_{H^{1}_{0}} + \|u_{k}\|^{2n}_{L^{2n}}\right) \left( |u_{k}|^{2}_{\mathcal{H}}-1\right)}
\end{align*}

Let $ \nu(t) = \left( |u_{k}(t)|^{2}_{\mathcal{H}}-1\right) $, implying that

\begin{align*}
     \frac{1}{2} \nu (t) &= \int_{0}^{t} {\left(\|u_{k}\|^{2}_{H^{2}_{0}} + 2 \|u_{k}\|^{2}_{H^{1}_{0}} + \|u_{k}\|^{2n}_{L^{2n}}\right) \nu (\mu)}~d\mu
\end{align*} 

\begin{align*}
      \nu '(t) &= 2{\left(\|u_{k}\|^{2}_{H^{2}_{0}} + 2 \|u_{k}\|^{2}_{H^{1}_{0}} + \|u_{k}\|^{2n}_{L^{2n}}\right) \nu (t)}
\end{align*} 
\begin{align*}
      \nu (t) &= \left( |u_{k}(0)|^{2}_{\mathcal{H}}-1\right) e^{2{\left(\|u_{k}\|^{2}_{H^{2}_{0}} + 2 \|u_{k}\|^{2}_{H^{1}_{0}} + \|u_{k}\|^{2n}_{L^{2n}}\right)}t}   
\end{align*}

As $|u_{k}(0)|_{\mathcal{H}} =\left|\frac{\pi_{k}u_{0}}{|\pi_{k}u_{0}|_{\mathcal{H}}}\right|_{\mathcal{H}}= 1$, we have

\begin{align*}
  |u_{k}(t)|^{2}_{\mathcal{H}}-1=0, ~~~~\forall~ t~ \in [0,T]
\end{align*} 

This result implies that $ u_{k}(t) \in \mathcal{M}$.\\

Next, we applied the same procedure to the energy norm.

\begin{align*}
    \frac{\|u_{k}\|^{2}_{\mathcal{V}}}{2}-  \frac{\|u_{0}\|^{2}_{\mathcal{V}}}{2} &=    {\|\nabla u_{k}\|^{2}_{L^{2}}} + \frac{\|\Delta u_{k}\|^{2}_{L^{2}}}{2}+ \frac{\| u_{k}\|^{2}_{L^{2}}}{2}-\frac{\|u_{0}\|^{2}_{L^{2}}}{2}-   {\|\nabla u_{0}\|^{2}_{L^{2}}} - \frac{\|\Delta u_{0}\|^{2}_{L^{2}}}{2}\\
\frac{d}{dt}\left(\frac{\|u_{k}(t)\|_{\mathcal{V}}^{2}}{2}\right)  &=  2{ \langle \nabla u_{k}(t), \nabla u_{k}'(t) \rangle} + { \langle \Delta u_{k}(t), \Delta u'_{k}(t) \rangle} +  { \langle u_{k}(t),  u'_{k}(t) \rangle}
\end{align*}  

Because $ u_{k}(t) \in  \mathcal{M}$,  $\langle u_{k}(t),  u'_{k}(t) \rangle = 0 $, it follows that
\begin{align*}
    \frac{d}{dt}\left(\frac{\|u_{k}(t)\|_{\mathcal{V}}^{2}}{2}\right)  &=  2{ \langle \nabla u_{k}(t), \nabla u_{k}'(t) \rangle} + { \langle \Delta u_{k}(t), \Delta u'_{k}(t) \rangle}  \\
    &=      - 2{ \langle \Delta u_{k}(t),  u'_{k}(t) \rangle} + { \langle \Delta^{2} u_{k}(t),  u'_{k}(t) \rangle}   \\
    &=  {\langle \Delta^{2} u_{k}(t) -2 \Delta u_{k}(t),  u'_{k}(t) \rangle}  \\
&=  {\langle \Delta^{2} u_{k}(t) -2 \Delta u_{k}(t) + u'_{k}(t),  u'_{k}(t) \rangle}  - {\langle u'_{k}(t), u'_{k}(t) \rangle } \\
\end{align*}

Using equation (\ref{approimate-problem}), we can infer that

\begin{align*}
   \frac{d}{dt}\left(\frac{\|u_{k}(t)\|_{\mathcal{V}}^{2}}{2}\right)   
    &= -{\left|\frac{du_{k}(t)}{dt} \right|_{\mathcal{H}}^{2}}  +  {\langle \|  u_{k}\|^{2}_{{H}^{2}_{0}} ~u_{k}(t) + 2\|    u_{k}\|^{2}_{{H}^{1}_{0}} ~u_{k} (t) +\| u_{k}\|^{2n}_{{L}^{2n}} u_{k}(t)- u_{k}(t)^{2n-1} ,  u'_{k}(t) \rangle}
\end{align*}

Because $ u_{k}(t) \in  \mathcal{M}$,  $\langle u_{k}(t),  u'_{k}(t) \rangle = 0 $, for all $t\in [0,T]$ it follows that

\begin{align}{\label{der_ene}}
    \frac{d}{dt}\left(\frac{\|u_{k}(t)\|_{\mathcal{V}}^{2}}{2}\right)  &= -{\left|\frac{du_{k}(t)}{dt} \right|_{\mathcal{H}}^{2}} - {\langle u_{k}(t)^{2n-1} ,  u'_{k}(t) \rangle}~ \notag \\
    \frac{d}{dt}\left(\Psi(u_{k}(t))\right)  &=-{\left|\frac{du_{k}(t)}{dt} \right|_{\mathcal{H}}^{2}} 
\end{align}

Where $ \Psi : \mathcal{V}\longrightarrow R $ is the energy norm, defined as 
\begin{align*}
    \Psi(u_{k}(t)) = \frac{1}{2} \|u_{k}(t)\|^{2}_{\mathcal{V}} + \frac{1}{2n} \|u_{k}(t)\|^{2n}_{L^{2n}}, ~~~~ n \in N 
\end{align*} 

From the above relation, it is clear that $\Psi$ is a non-increasing function and
\begin{align*}
    \Psi(u_{k}(t))\leq  \Psi(u_k{(0)}),~~~\forall t \in [0, T]
\end{align*}
Therefore, using (\ref{der_ene}), we can deduce that
\begin{align} {\label{sup-inq}}
   \|u_{k}(t)\|_{\mathcal{V}}\leq 2~\Psi(u_{k}(t))\leq  2~\Psi (u_{k}({0})), ~~~~~~\forall t \in [0,T] \notag \\
\sup_{t \in [0,T]}{\|u_{k}(t)\|_{\mathcal{V}}}\leq  2~\Psi(u_{k}{(0)})< \infty , ~~~~~~\forall t \in [0,T]
\end{align}

Thus, the sequence $u_{k}(t)$ is bounded in $\mathcal{L}^{\infty}(0,T; \mathcal{V})$

Again, 
\begin{align*}
     \frac{d}{dt}\left(\frac{\|u_{k}(t)\|_{\mathcal{V}}^{2}}{2}\right)  &={\langle \Delta^{2} u_{k}(t) -2 \Delta u_{k}(t),  u'_{k}(t) \rangle} \\
     &={\langle A_{k} u_{k}(t),  u'_{k}(t) \rangle} \\
     &={\langle A_{k} u_{k}(t), - A_{k} u_{k}(t)+ F_{k}(u_{k}(t) )\rangle}   \\
     &= -\left|A_{k} u_{k}(t)\right|^{2}_{\mathcal{H}} + \left(\|  u\|^{2}_{{H}^{2}_{0}}  + 2\|    u\|^{2}_{{H}^{1}_{0}}   +\| u\|^{2n}_{{L}^{2n}}\right) \|u\|^{2}_{\mathcal{V}} -{\langle A_{k} u_{k}(t),  u^{2n-1}_{k}(t) \rangle}
\end{align*}

By using Cauchy-Schwartz and Young’s inequalities, we can deduce that

\begin{align*}
     \frac{d}{dt}\left(\frac{\|u_{k}(t)\|_{\mathcal{V}}^{2}}{2}\right)  &\leq -\left|A_{k} u_{k}(t)\right|^{2}_{\mathcal{H}} + \left(\|  u\|^{2}_{{H}^{2}_{0}}  + 2\|    u\|^{2}_{{H}^{1}_{0}}   +\| u\|^{2n}_{{L}^{2n}}\right) \|u\|^{2}_{\mathcal{V}} +\frac{1}{2}{\left| A_{k} u_{k}(t)\right|^{2}_{\mathcal{H}} +   \left|u^{2n-1}_{k}(t) \right|^{2}_{\mathcal{H}}}\\
     &\leq -\frac{1}{2}\left|A_{k} u_{k}(t)\right|^{2}_{\mathcal{H}} + \left(\|  u\|^{2}_{{H}^{2}_{0}}  + 2\|    u\|^{2}_{{H}^{1}_{0}}   +\| u\|^{2n}_{{L}^{2n}}\right) \|u\|^{2}_{\mathcal{V}} + \left|u^{2n-1}_{k}(t) \right|^{2}_{\mathcal{H}} \\
     &= -\frac{1}{2}\left|A_{k} u_{k}(t)\right|^{2}_{\mathcal{H}} + \left(\|  u\|^{2}_{{H}^{2}_{0}}  + 2\|    u\|^{2}_{{H}^{1}_{0}}   +\| u\|^{2n}_{{L}^{2n}}\right) \|u\|^{2}_{\mathcal{V}} + \|u_{k}(t) \|^{4n-2}_{\mathcal{L}^{4n-2}}
\end{align*}

{For n, as described in equation (\ref{main_Prb_1})}, taking $a=1$, $p=4n-2$ and $p=2n$ for inequality (\ref{inq1}), we can deduce that there are constants $C_{1}$ and $C_{2}$ such that

\begin{align*}
     \frac{d}{dt}\left(\frac{\|u_{k}(t)\|_{\mathcal{V}}^{2}}{2}\right)  &\leq -\frac{1}{2}\left|A_{k} u_{k}(t)\right|^{2}_{\mathcal{H}} + \left(\|  u\|^{2}_{{H}^{2}_{0}}  + 2\|    u\|^{2}_{{H}^{1}_{0}}   +C^{2n}_{1}\| u\|^{2n}_{\mathcal{V}}\right) \|u\|^{2}_{\mathcal{V}} + C^{4n-2}_{2}\|u_{k}(t) \|^{4n-2}_{\mathcal{V}}
\end{align*}

 As $\mathcal{V} \hookleftarrow {{H}^{2}_{0}} $, $\mathcal{V} \hookleftarrow {{H}^{1}_{0}} $ and $\mathcal{V} \hookleftarrow {L^{2n}} $ is therefore $  \|  u\|^{2}_{{H}^{2}_{0}}  \leq k_{1}\|  u\|^{2}_{\mathcal{V}}$, and $  \|  u\|^{2}_{{H}^{1}_{0}}  \leq k_{2}\|  u\|^{2}_{\mathcal{V}}$. In addition, for all $t \in [0,T]$ using inequality (\ref{sup-inq}), it follows that

\begin{align*}
    \frac{d}{dt}\left(\frac{\|u_{k}(t)\|_{\mathcal{V}}^{2}}{2}\right)  &\leq -\frac{1}{2}\left|A_{k} u_{k}(t)\right|^{2}_{\mathcal{H}} +K
\end{align*}

Where $K= (k_{1}+2k_{2}+k_{3}C^{2n}_{1}) C^{4}+ C^{4n-2}_{2}C^{4n-2}$

By integrating from $0$ to $T$ on both sides, we have that

\begin{align}{\label{ieq_2}}
  \frac{1}{2}\int_{0}^{T}{\left|A_{k} u_{k}(t)\right|^{2}_{\mathcal{H}}}~dt \leq  \frac{\|u_{k}(t)\|_{\mathcal{V}}^{2}}{2} +\frac{1}{2}\int_{0}^{T}{\left|A_{k} u_{k}(t)\right|^{2}_{\mathcal{H}}}~dt &\leq  +\frac{\|u_{k}(0)\|_{\mathcal{V}}^{2}}{2} +KT
\end{align}

As $u_{k}(0) \in \mathcal{V}$ and using the above inequality, we can conclude that the sequence $u_{k}(t)$ is bounded in $\mathcal{L}^{2}(0,T; \mathcal{E})$. In addition, using inequalities (\ref{ieq_2}) and (\ref{sup-inq}), we can infer that sequence $u_{k}(t)$ is bound in the subset of $X_{T}$.

Finally, we prove that $u_{k}(t)$ is uniformly bounded in $\mathcal{L}^{2}(0,T; \mathcal{H})$ by considering the following equation:

\begin{align*}
   \int_{0}^{T}{{\left|\frac{du_{k}(t)}{dt} \right|_{\mathcal{H}}^{2}}}~dt =  \int_{0}^{T}{\left|A_{k} u_{k}(t)\right|^{2}_{\mathcal{H}}}~dt +  \int_{0}^{T}{\left|F_{k}(u_{k}(t))\right|^{2}}~dt + 2 \int_{0}^{T}{{\langle -A_{k} u_{k}(t),  F_{k}(u_{k}(t) )\rangle}}~dt
\end{align*}

Using  the Cauchy-Schwartz inequality followed by Young’s inequality, we can deduce that
\begin{align*}
     \int_{0}^{T}{{\left|\frac{du_{k}(t)}{dt} \right|_{\mathcal{H}}^{2}}}~dt &\leq  \int_{0}^{T}{\left|A_{k} u_{k}(t)\right|^{2}_{\mathcal{H}}}~dt +  \int_{0}^{T}{\left|F_{k}(u_{k}(t))\right|^{2}}~dt + 2 \int_{0}^{T}{\left|A_{k} u_{k}(t)\right|\left|F_{k}(u_{k}(t))\right|}~dt\\
     &\leq 2\int_{0}^{T}{\left|A_{k} u_{k}(t)\right|^{2}_{\mathcal{H}}}~dt +  3\int_{0}^{T}{\left|F_{k}(u_{k}(t))\right|^{2}}~dt
\end{align*}

Using inequalities (\ref{sup-inq}) and (\ref{ieq_2}), it follows that there are constants $L$ and $K$ such that

\begin{align}
    \int_{0}^{T}{{\left|\frac{du_{k}(t)}{dt} \right|_{\mathcal{H}}^{2}}}~dt &\leq 2 L+ 3\int_{0}^{T} K^{2n}dt\notag \\
    &= 2L + 3K^{2n}T < \infty
\end{align}

Therefore, we can conclude that the sequence $u'_{k}(t)$ is uniformly bounded in $\mathcal{L}^{2}(0,T; \mathcal{H})$. Using compactness theorem III (Aubin-Lion Lemma) \cite{zheng2004nonlinear}, it follows that
\begin{eqnarray}{\label{strong_con}}
    u_{k} \to u ~~~\text{strongly ~in }~~~\mathcal{L}^{2}(0,T; \mathcal{V})
\end{eqnarray}
\end{section}

\section{\textbf{ Well-posedness Of the Solution}}

In this section, we establish the main theorem of this study. That is, we intend to show that there is a unique solution to problem (\ref{main_prob}). Moreover, the continuity of the trajectory of the solution in the $\mathcal{V}$ norm is proven.

\begin{proof}

From the previous estimates (\ref{strong_con}), (\ref{ieq_2}), and (\ref{sup-inq}), we have shown that there is an element $u$ such that
\begin{eqnarray}
    u_{k} \rightarrow u, \quad &when ~k\to \infty ~ in ~\mathcal{L}^{2}(0,T; \mathcal{V}) ~~~strongly \notag \\
 u_{k} \rightarrow u, \quad &when ~k\to \infty ~ in ~\mathcal{L}^{2}(0,T; \mathcal{E}) ~~~weakly \notag\\
     u_{k} \rightarrow u, \quad &when ~k\to \infty ~ in ~\mathcal{L}^{\infty}(0,T; \mathcal{V}) ~~~weak^{*} ~sense \notag
\end{eqnarray}
In the following steps, we show that $u$ is the solution to Problem (\ref{main_prob}).

Assume that $\xi:[0,T] \to \mathcal{R}$ is a $C^{1}-$ class fixed function such that $\xi(T)=0$. Furthermore, $\phi \in H_{k}$ for a certain $k\in \mathbb{N}$. By multiplying the equation (\ref{approimate-problem}) with $\xi(\cdot) \phi$ and integrating w.r.t space variable, we have:

\begin{align*}
  \langle u'_{k} (t), \xi(t) \phi \rangle &= \left\langle -A_{k}u_{k}, \xi(t) \phi\rangle  + \langle F_{k}(u_{k}(t)), \xi(t) \phi \right\rangle
\end{align*}

Now, we integrate this from $0$ to $T$.

\begin{align*}
  \int_{0}^{T}\left\langle u'_{k} (t), \xi(t) \phi \right\rangle ~dt&=  \int_{0}^{T} \left\langle -A_{k}u_{k}, \xi(t) \phi \right \rangle~ dt +  \int_{0}^{T} \left\langle F_{k}(u_{k}(t)), \xi(t) \phi \right\rangle ~dt
\end{align*}

By integrating by parts on the left, we can deduce that

\begin{align}{\label{eq_int}}
  \int_{0}^{T}\left\langle u_{k} (t), \xi'(t) \phi \right\rangle ~dt&=  \int_{0}^{T} \left\langle A_{k}u_{k}, \xi(t) \phi \right \rangle~ dt -  \int_{0}^{T} \left\langle F_{k}(u_{k}(t)), \xi(t) \phi \right\rangle ~dt - \left\langle u_{0}, \xi(0) \phi \right\rangle 
\end{align}

To pass the limit $k \to \infty$, we investigate each term in Equation (\ref{eq_int}). \\
In our case, operator $A$ is defined as $A=\Delta^2-2\Delta: \mathcal{E} \rightarrow \mathcal{H}$.

 The pseudo-inverse of $A$, \(A^{\dagger}\), acts as a generalized inverse for \({A}\) in \(\mathcal{H}\) and can be defined similarly based on the spectral properties of \({A}\).

Let $\mathcal{H}$ have  an orthonormal basis $\{e_{j}\}$ consisting of eigenvectors corresponding to the eigenvalues \(\{\lambda_j\}\) such that

\[
\mathcal{A} e_j = \lambda_j e_j, \quad j =1,2,3,...k
\]

The pseudo-inverse \({A}^{\dagger}: \mathcal{H} \to D(\mathcal{A})\) is defined as

\begin{eqnarray}
    {A}^{\dagger} e_j =
\begin{cases}
\frac{1}{\lambda_j} e_j, & \text{if } \lambda_j \neq 0, \\
0 & \text{if } \lambda_j = 0.
\end{cases}
\end{eqnarray}

More precisely, for any \(u\in \mathcal{H}\),  \({A}^{\dagger}\) on \(u\) can be expressed as:

\begin{eqnarray}
    {A}^{\dagger} u = \sum_{j : \lambda_j \neq 0} \frac{\langle u, e_j \rangle}{\lambda_j} e_j,
\end{eqnarray}

where \(\langle u, e_j \rangle\) denotes the inner product of \(\mathcal{H}\).

$A$ is a self-adjoint operator (\ref{Self-Ad_op}); therefore, the inner product on $\mathcal{L}^{2}(0,T; \mathcal{E}) $ can be defined as
\begin{align}
 \langle u, v\rangle_{\mathcal{L}^{2}(0,T; \mathcal{E}) }&= \int_{0}^{T}{\langle Au(p), Av(p)\rangle}_{\mathcal{H}}~dp~~~~~~~~or \notag \\
 \langle A^{\dagger} u, A^{\dagger} v\rangle_{\mathcal{L}^{2}(0,T; \mathcal{E}) }&= \int_{0}^{T}{\langle u(p), v(p)\rangle}_{\mathcal{H}}~dp
\end{align}

Let us consider the first term of equation (\ref{eq_int}) on the right-hand side.

\begin{align*}
    \int_{0}^{T} \left\langle A_{k}u_{k}, \xi(t) \phi \right \rangle~ dt &= \int_{0}^{T} \left\langle \pi_{k}Au_{k}, \xi(t) \phi \right \rangle~ dt \\
    &= \int_{0}^{T} \left\langle Au_{k}, \xi(t) \pi_{k}\phi \right \rangle~ dt
\end{align*}

For $k \in \mathbb{N}$ and $\phi \in \mathcal{H}_{k}$ it is clear that $\pi_{k}\phi =\phi $. It follows that:

\begin{align}{\label{frst-term}}
   \int_{0}^{T} \left\langle A_{k}u_{k}, \xi(t) \phi \right \rangle~ dt &= \int_{0}^{T} \left\langle Au_{k}, \xi(t) \phi \right \rangle~ dt \notag \\
   &=\int_{0}^{T} \left\langle u_{k}, A^{\dagger}\xi(t) \phi \right \rangle_{D(A)}~ dt
\end{align}

Now, using $A^{\dagger}\xi(t) \phi \in \mathcal{L}^{2}(0,T; D(A))$, $u_{k}\to u$ in $ \mathcal{L}^{2}(0,T; D(A))$ and equation  (\ref{frst-term}), we have

\begin{align*}
    \int_{0}^{T} \left\langle A_{k}u_{k}, \xi(t) \phi \right \rangle_{}~ dt -\int_{0}^{T} \left\langle u_{}, A^{\dagger}\xi(t) \phi \right \rangle_{D(A)}~ dt = \int_{0}^{T} \left\langle u_{k}-u, A^{\dagger}\xi(t) \phi \right \rangle_{D(A)}~ dt \rightarrow 0~~~ as ~~k \to \infty 
\end{align*}

Consider the second term on the equation (\ref{eq_int}) on the right side.

\begin{align*}
    \int_{0}^{T} \left\langle F_{k}(u_{k}(t)), \xi(t) \phi \right\rangle ~dt &=   \int_{0}^{T} \left\langle \pi_{k}F_{}(u_{k}(t)), \xi(t) \phi \right\rangle ~dt \\
    &=\int_{0}^{T} \left\langle F_{}(u_{k}(t)), \xi(t) \pi_{k}\phi \right\rangle ~dt\\
    &= \int_{0}^{T} \left\langle F_{}(u_{k}(t)), \xi(t) \phi \right\rangle ~dt
\end{align*}

Using  the Cauchy-Schwartz inequality, the following can be deduced:

\begin{align*}
    \left|\int_{0}^{T} \left\langle F_{k}(u_{k}(t))-F_{}(u_{k}(t)), \xi(t) \phi \right\rangle ~dt\right| \leq  \sqrt { \int_{0}^{T}  {\left|F_{k}(u_{k}(t))-F(u_{k}(t))\right|_{\mathcal{H}}^{2}}~dt  }~\sqrt{\int_{0}^{T}  {\left|\xi(t) \phi\right|_{\mathcal{H}}^{2}}~dt }
 \end{align*}

Then, $\sqrt { \int_{0}^{T}  {\left|F_{k}(u_{k}(t))-F(u_{k}(t))\right|_{\mathcal{H}}^{2}}~dt  }~\sqrt{\int_{0}^{T}  {\left|\xi(t) \phi\right|_{\mathcal{H}}^{2}}~dt } \to 0$ when $k \to \infty$. It is sufficient to prove that $ \int_{0}^{T}  {\left|F_{k}(u_{k}(t))-F(u_{k}(t))\right|_{\mathcal{H}}^{2}}  \to 0$ as $k \to \infty$.\\

In addition, we used the following Cauchy-Schwartz inequality:
\begin{eqnarray}
    \left(\sum_{j=1}^{i}x_{j}\right)^{2} = \left(\sum_{j=1}^{i} 1\cdot x_{j}\right)^{2} \leq i\cdot \sum_{j=1}^{i}x^{2}_{j}
\end{eqnarray}
\begin{align*}
    \int_{0}^{T}  {\left|F(u_{k}(t))-F(u(t))\right|_{\mathcal{H}}^{2}} &\leq K \int_{0}^{T}
    \begin{bmatrix} \left(\|u_{k}\|_{\mathcal{V}}^{2}+\|u_{}\|_{\mathcal{V}}^{2}+\|u_{k}\|_{\mathcal{V}}\|u_{}\|_{\mathcal{V}}\right) \\ +  \left( \|u_{k}\|_{\mathcal{V}}^{2n-1}+\|u_{}\|_{\mathcal{V}}^{2n-1}\right) (\|u_{k}\|_{\mathcal{V}}+\|u_{}\|_{\mathcal{V}}) \\ +  \left( \|u_{k}\|_{\mathcal{V}}^{2n}+\|u_{}\|_{\mathcal{V}}^{2n}\right)+ \left( 1+\|u_{k}\|_{\mathcal{V}}^{2}+\|u_{}\|_{\mathcal{V}}^{2}\right)^{\frac{1}{3}} 
    \end{bmatrix}^{2}  \|u_{k}-u_{}\|^{2}_\mathcal{V}~dt \\
    &\leq 4K \int_{0}^{T}
    \begin{bmatrix} \left(\|u_{k}\|_{\mathcal{V}}^{2}+\|u_{}\|_{\mathcal{V}}^{2}+\|u_{k}\|_{\mathcal{V}}\|u_{}\|_{\mathcal{V}}\right)^{2} \\ +  \left( \|u_{k}\|_{\mathcal{V}}^{2n-1}+\|u_{}\|_{\mathcal{V}}^{2n-1}\right)^{2} (\|u_{k}\|_{\mathcal{V}}+\|u_{}\|_{\mathcal{V}})^{2} \\ +  \left( \|u_{k}\|_{\mathcal{V}}^{2n}+\|u_{}\|_{\mathcal{V}}^{2n}\right)^{2}+ \left( 1+\|u_{k}\|_{\mathcal{V}}^{2}+\|u_{}\|_{\mathcal{V}}^{2}\right)^{\frac{2}{3}} 
    \end{bmatrix}  \|u_{k}-u_{}\|^{2}_\mathcal{V}~dt\\
     &\leq 4K \int_{0}^{T}
    \begin{bmatrix} \left(M^{2}+N^{2}+MN\right)^{2} \\ +  \left( M^{2n-1}+N^{2n-1}\right)^{2} (M+N)^{2} \\ +  \left( M^{2n}+N^{2n}\right)^{2}+ \left(1+ M^{2}+N^{2}\right)^{\frac{2}{3}} 
    \end{bmatrix}  \|u_{k}-u_{}\|^{2}_\mathcal{V}~dt \\
    &\leq L ~\|u_{k}-u_{}\|^{2}_{\mathcal{L}^{2}(0,T; \mathcal{V})} \to 0 ~~~ \text{as}~~~~ k \to \infty,
    \end{align*}

where 
\begin{align*}
    M = \sup_{t \in [0,T]}{\|u_{k}(t)\|^{2}}, ~~N&= \sup_{t \in [0,T]}{\|u_{}(t)\|^{2}} ~~\text{as both}~~~ u_{k}, ~u \in \mathcal{L}^{\infty}(0,T; \mathcal{V}) ~~ \text{so they have finite value.}\\
    L&:= 4K \begin{bmatrix} \left(M^{2}+N^{2}+MN\right)^{2}  +  \left( M^{2n-1}+N^{2n-1}\right)^{2} (M+N)^{2} \\ +  \left( M^{2n}+N^{2n}\right)^{2}+ \left(1+ M^{2}+N^{2}\right)^{
    \frac{2}{3}
    } 
    \end{bmatrix} < \infty 
\end{align*}

Finally, we consider $ \int_{0}^{T}\left\langle u_{k} (t), \xi'(t) \phi \right\rangle ~dt$ in the inequality (\ref{eq_int}).

\begin{align}
     \int_{0}^{T}\left\langle u_{k} (t)-u (t), \xi'(t) \phi \right\rangle ~dt &\leq  \int_{0}^{T} \|u_{k} (t)-u (t)\|_{\mathcal{V}}~\|\xi'(t)\phi \|_{\mathcal{V}}  \notag \\
     &\leq \sqrt{\int_{0}^{T}{\|u_{k} (t)-u (t)\|_{\mathcal{V}}^{2}}~dt}~~\sqrt{\int_{0}^{T}{\|\xi'(t)\phi \|_{\mathcal{V}} ^{2}}~dt}\to 0 ~~~\text{as}~~ k \to \infty
\end{align}

Therefore, we can pass the limit on to Equation (\ref{eq_int}). That is:

\begin{align}{\label{eq1}}
  \int_{0}^{T}\left\langle u(t), \xi'(t) \phi \right\rangle ~dt&=  \int_{0}^{T} \left\langle Au(t), \xi(t) \phi \right \rangle~ dt -  \int_{0}^{T} \left\langle F(u(t)), \xi(t) \phi \right\rangle ~dt - \left\langle u_{0}, \xi(0) \phi \right\rangle\\
  ~~~~~ \text{For all}~~ \phi \in \cup_{k=1} H_{k} \notag 
\end{align}

Because $\cup_{k=1} H_{k}$ is dense in $\mathcal{V}$, (\ref{eq1}) holds for any $\phi \in \mathcal{V}$ and $\xi \in C_{0}^{1}\left([0,T]\right)$. Therefore, $u$ satisfies evolution equation (\ref{PB}) in $\mathcal{L}^{2}(0,T; \mathcal{H})$.

Let us now consider the equation for the initial value. Consider $\phi \in \mathcal{V}$ and $\xi \in C_{0}^{1}\left([0,T]\right)$ such that $\xi(0) =1$. Multiplying $\xi(t)\phi$ by equation (\ref{PB}) and applying integration by parts, we obtain

\begin{align}{\label{eq2}}
     \int_{0}^{T}\left\langle u(t), \xi'(t) \phi \right\rangle ~dt&=  \int_{0}^{T} \left\langle Au(t), \xi(t) \phi \right \rangle~ dt -  \int_{0}^{T} \left\langle F(u(t)), \xi(t) \phi \right\rangle ~dt - \left\langle u(0), \xi(0) \phi \right\rangle
\end{align}
By comparing equations (\ref{eq1}) and (\ref{eq2})  and using the fact that $\xi(0)=1$ we can deduce that

\begin{align}
    \left\langle u(0)-u_{0},  \phi \right\rangle =0,~~~~~~\text{for all}~~~ \phi \in \mathcal{V}
\end{align}

As $\mathcal{V}$ is dense in $\mathcal{H}$, we can infer that $u(0)-u_{0}=0$ which implies $u(0)=u_{0}$. Therefore, $u$ solves problem (\ref{PB}).

We have already proved that $u \in \mathcal{L}^{2}(0,T; \mathcal{H})$ in (\ref{strong_con}). For $u \in X_{T}$, we must show that $ \frac{\partial u}{ \partial t} \in \mathcal{L}^{2}(0,T; \mathcal{H})$. As we have shown that $u$ solves the problem (\ref{PB}), so consider the following equation 

\begin{align*}
    \frac{\partial u}{\partial t} = -Au + F(u(t))
\end{align*}

Now, by taking the inner product in $\mathcal{H}$, integrating the equation from $0$ to $T$ on both sides, and applying Minkowski’s inequality for $p = 2$,  we can deduce that

\begin{align}{\label{u_{t}}}
     \int_{0}^{T}\left|\frac{\partial u}{\partial t} \right|^{2}_{\mathcal{H}}~dt\leq   \left(\int_{0}^{T} {\left|Au(t)\right|_{\mathcal{H}}^{2}}~dt\right)^{\frac{1}{2}} +\left(\int_{0}^{T} {\left|F(u(t))\right|_{\mathcal{H}}^{2}}~dt\right)^{\frac{1}{2}}  \leq |u|^{}_{\mathcal{L}^{2}(0,T; \mathcal{E})} + \left(\int_{0}^{T} {\left|F(u(t))\right|_{\mathcal{H}}^{2}}~dt\right)^{\frac{1}{2}}
\end{align}

We know that $u \in \mathcal{L}^{2}(0,T; \mathcal{E})$ referred to in (\ref{ieq_2}), which implies that $|u|^{}_{\mathcal{L}^{2}(0,T; \mathcal{E})} < \infty$.\\
Let us consider the second term $\left(\int_{0}^{T} {\left|F(u(t))\right|_{\mathcal{H}}^{2}}~dt\right)^{\frac{1}{2}}$.

\begin{align*}
    \int_{0}^{T} {\left|F(u(t))\right|_{\mathcal{H}}^{2}}~dt &\leq 4K \int_{0}^{T}
    \begin{bmatrix} \left(\|u_{}\|_{\mathcal{V}}^{2}+\|u_{}\|_{\mathcal{V}}\right)^{2}  + 2 \|u_{}\|_{\mathcal{V}}^{4n}  + \|u_{}\|_{\mathcal{V}}^{4n-4} 
    \end{bmatrix}  \|u_{}\|^{2}_\mathcal{V}~dt \\
    &\leq  4K \int_{0}^{T} \begin{bmatrix} \left(N^{2}+N\right)^{2}  +  2N^{4n} + N^{4n-4} 
    \end{bmatrix}  N^{2}~dt \leq CT < \infty 
\end{align*}
Where $ N= \sup_{t \in [0,T]}{\|u_{}(t)\|_{\mathcal{V}}^{2}} < \infty ~~\text{and}~~~
    C:= 4K  \begin{bmatrix} \left(N^{2}+N\right)^{2}  +  2N^{4n} + N^{4n-4} 
    \end{bmatrix}  N^{2} < \infty$ because $u \in \mathcal{L}^{2}(0,T; \mathcal{V})$. Thus, from inequality (\ref{u_{t}}), we can deduce that $ \frac{\partial u}{ \partial t} \in \mathcal{L}^{2}(0,T; \mathcal{H}) $, and thus, $u \in X_{T}$.

{\large \textbf{Uniqueness of the solution} }\\

Let $\Lambda (t)=u_{1}(t)-u_{2(t)} \in X_{T}$ and $\Lambda(0)= u_{1}(0)-u_{2}(0) = 0$ and $\Lambda(t)$ solve problem (\ref{PB}). In addition, assume that $u_{1}$, $u_{1} \in \mathcal{L}^{2}(0,T; \mathcal{H})$ are the two solutions of problem (\ref{PB}), so it follows that
\begin{align*}
     \frac{\|\Lambda\|^{2}_{\mathcal{V}}}{2}-  \frac{\|\Lambda_{0}\|^{2}_{\mathcal{V}}}{2} &= -\int^{t}_{0} {\left\|\Lambda'(p) \right\|_{L^{2}}^{2}} ~dp+ \int^{t}_{0} {\langle \Delta^{2} \Lambda(p) -2 \Delta \Lambda(p) + \Lambda'_{}(p),  \Lambda'_{}(p) \rangle}~dp   \\    
\end{align*}
By using the Cauchy-Schwartz inequality, it follows that
\begin{align*}
    \frac{\|\Lambda\|^{2}_{\mathcal{V}}}{2}-  \frac{\|\Lambda_{0}\|^{2}_{\mathcal{V}}}{2} &\leq   -\int^{t}_{0} {\left\|\Lambda'_{}(p) \right\|_{L^{2}}^{2}} ~dp + \int^{t}_{0} \left( \int_{D}|\Delta^{2} \Lambda(p) -2 \Delta \Lambda(p) + \Lambda'_{}(p)|^{2} ~dp\right)^{\frac{1}{2}}{\left(\int_{D}|\Lambda'_{}(p)|^{2} ~dp\right)^{\frac{1}{2}}   }~dp \\
    & \leq  -\int^{t}_{0} {\left\|\Lambda'_{}(p) \right\|_{L^{2}}^{2}} ~dp + \int^{t}_{0} {\|\Delta^{2} \Lambda(p) -2 \Delta \Lambda(p) + \Lambda'_{}(p)\|_{L^{2}}}{\|\Lambda'_{}(p)\|_{L^{2}}} ~dp  
\end{align*}

Using the fact that, $ ab \leq \frac{a^{2}+b^{2}}{2}$, it follows:
\begin{align*}
     \frac{\|\Lambda\|^{2}_{\mathcal{V}}}{2}-  \frac{\|\Lambda_{0}\|^{2}_{\mathcal{V}}}{2} &\leq   -\int^{t}_{0} {\left\|\Lambda'_{}(p) \right\|_{L^{2}}^{2}} ~dp + \frac{1}{2}\int^{t}_{0} {\|\Delta^{2} \Lambda(p) -2 \Delta \Lambda(p) + \Lambda'_{}(p)\|^{2}_{L^{2}}}~dp + \frac{1}{2} \int_{0}^{t}{\|\Lambda'_{}(p)\|^{2}_{L^{2}}} ~dp 
\end{align*}

    As $u_{0}(0)=0$ and $\|\Lambda_{0}\|_{\mathcal{V}} =0 $, it follows that
    \begin{align*}
       \frac{\|\Lambda\|^{2}_{\mathcal{V}}}{2} &\leq   -\int^{t}_{0} {\left\|\Lambda'_{}(p) \right\|_{L^{2}}^{2}} ~dp + \frac{1}{2}\int^{t}_{0} {\|\Delta^{2} \Lambda(p) -2 \Delta \Lambda(p) + \Lambda'_{}(p)\|^{2}_{L^{2}}}~dp+ \frac{1}{2} \int_{0}^{t}{\|\Lambda'_{}(p)\|^{2}_{L^{2}}} dp  \\
       & \leq - \frac{1}{2}\int^{t}_{0} {\left\|\Lambda'_{}(p) \right\|_{L^{2}}^{2}} ~dp + \frac{1}{2}\int^{t}_{0} {\|\Delta^{2} \Lambda(p) -2 \Delta \Lambda(p) + \Lambda'_{}(p)\|^{2}_{L^{2}}}~dp\\
      \|\Lambda\|^{2}_{\mathcal{V}} +  \int^{t}_{0} {\left\|\Lambda'_{}(p) \right\|_{L^{2}}^{2}} ~dp &\leq  \int^{t}_{0} {\|\Delta^{2} \Lambda(p) -2 \Delta u(p) + \Lambda'_{}(p)\|^{2}_{L^{2}}}~dp  \leq \int_{0}^{t}{\|F(u_{1})-F(u_{2})\|^{2}}~dp
    \end{align*} 

But, $\|F(u_{1})-F(u_{2})\| \leq K
    \begin{bmatrix} \left(\|u_{1}\|_{\mathcal{V}}^{2}+\|u_{2}\|_{\mathcal{V}}^{2}+\|u_{1}\|_{\mathcal{V}}\|u_{2}\|_{\mathcal{V}}\right) \\ +  \left( \|u_{1}\|_{\mathcal{V}}^{2n-1}+\|u_{2}\|_{\mathcal{V}}^{2n-1}\right) (\|u_{1}\|_{\mathcal{V}}+\|u_{2}\|_{\mathcal{V}}) \\ +  \left( \|u_{1}\|_{\mathcal{V}}^{2n}+\|u_{2}\|_{\mathcal{V}}^{2n}\right)+ \left(1+ \|u_{1}\|_{\mathcal{V}}^{2}+\|u_{2}\|_{\mathcal{V}}^{2}\right)^{\frac{1}{3}} 
    \end{bmatrix}  \|u_{1}-u_{2}\|_\mathcal{V}$, it follows:
\begin{align*}
     \|\Lambda\|^{2}_{\mathcal{V}} +  \int^{t}_{0} {\left\|\Lambda'_{}(p) \right\|_{L^{2}}^{2}} ~dp &\leq  K\int_{0}^{t}{
    \begin{bmatrix} \left(\|u_{1}\|_{\mathcal{V}}^{2}+\|u_{2}\|_{\mathcal{V}}^{2}+\|u_{1}\|_{\mathcal{V}}\|u_{2}\|_{\mathcal{V}}\right) \\ +  \left( \|u_{1}\|_{\mathcal{V}}^{2n-1}+\|u_{2}\|_{\mathcal{V}}^{2n-1}\right) (\|u_{1}\|_{\mathcal{V}}+\|u_{2}\|_{\mathcal{V}}) \\ +  \left( \|u_{1}\|_{\mathcal{V}}^{2n}+\|u_{2}\|_{\mathcal{V}}^{2n}\right)+ \left( 1+\|u_{1}\|_{\mathcal{V}}^{2}+\|u_{2}\|_{\mathcal{V}}^{2}\right)^{\frac{1}{3}}
    \end{bmatrix}^{2}  }\|u_{1}-u_{2}\|_{\mathcal{V}}^{2}~dp\\
     & \leq  K \int_{0}^{t}{\begin{bmatrix} \left(\|u_{1}\|_{\mathcal{V}}^{2}+\|u_{2}\|_{\mathcal{V}}^{2}+\|u_{1}\|_{\mathcal{V}}\|u_{2}\|_{\mathcal{V}}\right) \\ +  \left( \|u_{1}\|_{\mathcal{V}}^{2n-1}+\|u_{2}\|_{\mathcal{V}}^{2n-1}\right) (\|u_{1}\|_{\mathcal{V}}+\|u_{2}\|_{\mathcal{V}}) \\ +  \left( \|u_{1}\|_{\mathcal{V}}^{2n}+\|u_{2}\|_{\mathcal{V}}^{2n}\right)+ \left( 1+\|u_{1}\|_{\mathcal{V}}^{2}+\|u_{2}\|_{\mathcal{V}}^{2}\right)^{\frac{1}{3}} 
    \end{bmatrix}^{2}}\|\Lambda\|_{\mathcal{V}}^{2}~dp
\end{align*}
Using Gronwall's inequality, we can deduce that

\begin{align*}
    \|\Lambda(t)\|^{2}_{\mathcal{V}} &\leq 0 \cdot \exp\left(\begin{bmatrix} \left(\|u_{1}\|_{\mathcal{V}}^{2}+\|u_{2}\|_{\mathcal{V}}^{2}+\|u_{1}\|_{\mathcal{V}}\|u_{2}\|_{\mathcal{V}}\right) \\ +  \left( \|u_{1}\|_{\mathcal{V}}^{2n-1}+\|u_{2}\|_{\mathcal{V}}^{2n-1}\right) (\|u_{1}\|_{\mathcal{V}}+\|u_{2}\|_{\mathcal{V}}) \\ +  \left( \|u_{1}\|_{\mathcal{V}}^{2n}+\|u_{2}\|_{\mathcal{V}}^{2n}\right)+ \left( 1+\|u_{1}\|_{\mathcal{V}}^{2}+\|u_{2}\|_{\mathcal{V}}^{2}\right) ^{\frac{1}{3}}
    \end{bmatrix}\right)^{2}=0\\
    \|\Lambda(t)\|^{2}_{\mathcal{V}}&=0 \iff \|\Lambda(t)\|_{\mathcal{V}}=0 \iff 
    u_{1}(t)=u_{2}(t)
\end{align*}

Thus, the solution to Problem (\ref{PB}) is unique.

\end{proof}

\section{Conclusion}

This study aims to investigate the nonlinear deterministic constrained modified Swift–Hohenberg equation. In this work, we have shown the existence and uniqueness of the solution to the proposed problem using the Faedo-Galerkin Compactness method. In our study, we were able to find a solution restricting $u\in \mathcal{V} \cap \mathcal{M}$, but the solution $u\in \mathcal{V} $ is still open to further investigation, as it is difficult to eliminate a constant $a$ from the equation $(\ref{SHEq}) $ for such a case.\\

\textbf{Declarations}\\ \\
\textbf{Conflict of Interest:} We confirm that this manuscript is the author's original work, has not been published, and is not under consideration for publication elsewhere. The authors declare that they have no known competing financial interests or personal relationships that could have influenced the work reported in this paper.\\\\
\textbf{Data Availability:} No data were required for this research.\\\\
\textbf{Author Contributions:} All authors have contributed equally to this manuscript.\\\\
\textbf{Ethical Approval:} This article contains no studies with human participants or animals performed by authors.\\\\
\textbf{Informed Consent:} This article contains no studies with human participants. Just to let you know, ent is not applicable here.

\end{document}